\documentclass[12pt]{article}
\usepackage{graphicx}
\usepackage{xcolor}
\usepackage{amssymb,amsfonts,amsthm,amsmath,calligra,tikz}
\usepackage[utf8]{inputenc}
\usepackage{accents}
\usepackage{slashed}
\usepackage{yfonts}
\usepackage{mathrsfs,pifont}
\theoremstyle{definition}
\usepackage[all]{xy}
\usepackage{enumitem}

\usepackage[style=numeric]{biblatex}
\def\be{\begin{eqnarray}}
\def\ee{\end{eqnarray}}

\tikzstyle{vertex}=[circle, draw, inner sep=0pt, minimum size=6pt]

\def\matQ{{\mathbb{Q}}}
\def\matC{{\mathbb{C}}}
\def\matN{{\mathbb{N}}}
\newcommand{\bA}{\mathsf{A}}

\newcommand{\bT}{\mathsf{T}}
\newcommand{\bK}{\mathsf{K}}

\newcommand{\Lie}{\mathrm{Lie}}

\let\bs\boldsymbol

\def\zz{{\bs z}}

\def\aa{{\bs a}}

\def\ss{{\bs s}}

\theoremstyle{definition}
\newtheorem{Definition}{Definition}
\newtheorem{Proposition}{Proposition}
\newtheorem{Lemma}{Lemma}

\newtheorem{Theorem}{Theorem}

\newtheorem{Note}{Note}

\makeatletter
\newcommand{\somespecialrotate}[3][]{%
\begingroup
\sbox\@tempboxa{#3}%
\@tempdima=.5\wd\@tempboxa
\sbox\@tempboxa{\rotatebox[#1]{#2}{\usebox\@tempboxa}}%
\advance\@tempdima by -.5\wd\@tempboxa
\mbox{\hskip\@tempdima\usebox\@tempboxa}%
\endgroup}
\linethickness{1 pt}

\def\qm {{{\textnormal{\textsf{QM}}}}}

\def \vss {\widehat{{\O}}_{{\rm{vir}}}}

\def\zz{{\bs z}}
\def\dd{{\bs d}}
\def\xx{{\bs x}}

\def\O {{\mathcal{O}}}

\begin{document}
\title{Symplectic Duality for $T^*Gr(k,n)$}
\author{Hunter Dinkins}
\date{}
\maketitle
\thispagestyle{empty}

\begin{abstract}
In this paper, we explore a consequence of symplectic duality (also known as 3d mirror symmetry) in the setting of enumerative geometry. The theory of quasimaps allows one to associate hypergeometric functions called vertex functions to quiver varieties. In this paper, we prove a formula which relates the vertex functions of $T^*Gr(k,n)$ and its symplectic dual. In the course of the proof, we study a family of $q$-difference operators which act diagonally on Macdonald polynomials. Our results may be interpreted from a combinatorial perspective as providing an evaluation formula for a $q$-Selberg type integral. 
\end{abstract}

\section{Introduction}
\subsection{}
The concept of symplectic duality, sometimes referred to as 3D mirror symmetry, originated in physics and has been attracting increasing attention from mathematicians in recent years, see \cite{AOElliptic}, \cite{dinksmir3}, \cite{dinksmir}, \cite{MirSym2}, \cite{MirSym1}. For a variety $X$ with certain conditions, it is expected that there exists a symplectic dual variety, which we denote by $X^{!}$, so that many deep geometric properties of $X$ and $X^{!}$ are related. From a physical point of view, symplectic duality exchanges the Higgs and Coulomb branches of certain three dimensional gauge theories. From a mathematical perspective, it has led to deep results and expectations in topics of enumerative geometry, stable envelopes, and quantum difference equations, all of which are important topics in geometric representation theory. In some cases, such as bow varieties \cite{NakBow}, the construction of the symplectic dual variety is known. Once one has a candidate for the symplectic dual of a variety $X$, it is generally a nontrivial problem to verify that the expected mathematical properties hold.

In this paper, we consider the variety $X=T^*Gr(k,n)$ and its proposed dual variety $X^{!}$, both of which can be constructed as Nakajima quiver varieties when $2k\leq n$. In \cite{MirSym1}, the authors prove that that expected relationship between the elliptic stable envelopes of $X$ and $X^{!}$ hold. Here we take a different approach and explore symplectic duality from the perspective of enumerative geometry. The main result of this paper is that, after appropriate normalizations and identification of various parameters, the enumerative invariants of $X$ and $X^{!}$ known as \textit{vertex functions} coincide.

The basic ideas of vertex functions and the enumerative geometry of Nakajima quiver varieties are as follows. For a geometric invariant theory quotient $X$ with certain conditions which are satisfied by Nakajima quiver varieties, see \cite{qm} and \cite{pcmilect}, one can define the moduli space of quasimaps from $\mathbb{P}^1$ to $X$ of degree $\dd$. This space compactifies the space of maps from $\mathbb{P}^1$ to $X$ by allowing the maps to have singularities at finitely many points. The maximal torus $\bT \subset Aut(X)$ and an additional torus $\mathbb{C}^{\times}_q$ both act on the space of quasimaps of a given degree.

On the moduli space of quasimaps of degree $\dd$, there exists a certain natural $K$-theory class $\vss^{\dd}$ called the symmetrized virtual structure sheaf. For quasimaps nonsingular at $\infty \in \mathbb{P}^1$, one can use equivariant localization with respect to $\bT\times \mathbb{C}^{\times}_q$ to pushforward $\vss^{\dd}$ by the evaluation map at $\infty$ to obtain a $K$-theory class on $X$. The generating function
$$
\textbf{V}(\zz) = \sum_{\dd} \text{ev}_{\infty,*}(\vss^{\dd}) \zz^{\dd} \in K_{\bT\times \mathbb{C}^{\times}_q}(X)[[\zz]]
$$
is known as the vertex function of $X$. Here, the parameter $\zz$ is introduced formally to keep track of the degrees of the quasimaps and the sum is taken over a certain cone inside of $\mathbb{Z}^{|I|}$ where $I$ is the vertex set of the quiver, outside of which the quasimap moduli spaces are empty. Restricting to a torus fixed point $p \in X^{\bT}$ gives a power series 
$$
\textbf{V}_{p}(\zz) \in K_{\bT\times \mathbb{C}^{\times}_q}(p)[[\zz]]
$$
Since the pushfoward is defined by equivariant localization, the function $\textbf{V}_p(\zz)$ is a power series in $\zz$, with coefficients in $\mathbb{C}(\aa,q)$, where $\aa$ denotes the equivariant parameters of the torus $\bT$. 

It is known that the vertex functions $\textbf{V}_p(\zz)$, $p \in X^{\bT}$ give a basis of solutions to a certain system of $q$-difference equations in $\zz$, see \cite{OS} and \cite{AOElliptic}. For a pair of symplectic dual varieties $X$ and $X^{!}$, it is expected that the systems of $q$-differences equations are equivalent, after a certain identification of the parameters $\zz$ with the equivariant parameters of $X^{!}$ and vice versa. As a result, the vertex functions for $X$ and $X^{!}$ can be thought of as giving two bases of solutions to the same $q$-difference equation. It is also expected that, after appropriate normalization, the so-called elliptic stables envelope provides the transition matrix between these two bases of solutions. As the elliptic stable envelope gives a triangular matrix in the basis of fixed points, it is expected that the vertex functions for $X$ and $X^{!}$ corresponding to the ``last" fixed points in some ordering coincide after identification of parameters and normalization by some factor. 

Since the identification of the parameters involves interchanging the degree counting parameters $\zz$ with the equivariant parameters on the dual side and vice versa, this equality is equivalent to a nontrivial combinatorics involving re-expanding a power series in a different set of variables and collecting certain terms into rational functions.

As the main result of this paper, we prove that the vertex functions at the ``last" fixed point for the variety $X=T^*Gr(k,n)$ and its proposed dual $X^{!}$ are equal, in the case of $2k\leq n$. The main ideas of the proof are as follows. 

From the integral representation of the vertex function in Section 4, it follows that the vertex function for a Nakajima quiver variety can often be thought of as a \textit{descendant insertion} into the vertex function of simpler quiver variety. In this case, the simpler quiver variety is just a point, the vertex function of which was explored in detail in \cite{dinksmir2} and \cite{dinksmir3}. As observed in \cite{dinksmir3}, descendant insertions into vertex functions can in some cases be described through the action of certain $q$-difference operators. For $X$ and $X^{!}$ in this paper, we are led to introduce a family of $q$-difference operators commuting with the Macdonald difference operators, closely related to those studied in \cite{NoumiSano} and \cite{NSmac}. Once the fundamental property of the difference operators is known, the remainder of the proof can be completed by direct computation.

The structure of this paper is as follows. In Section 2, we introduce the descriptions of $X=T^*Gr(k,n)$ and its dual $X^{!}$ as Nakajima quiver varieties. We will also define another variety $X_{\lambda}$, whose quiver data and vertex function are closely related to that of $X^{!}$. Then we describe the torus fixed points of $X$ and $X^{!}$ and a bijection between them. 

In Section 3, we define the vertex functions of quiver varieties and describe the integral representation for type $A$ quiver varieties. We also explain the meaning of descendant insertions.

In Section 4, we give explicit formulas for the vertex functions at the ``last" fixed point, define the identification of parameters, and precisely state our main result.

In order to study the relationship between the vertex functions of $X$ and $X^{!}$, we are led in Section 5 to introduce the following family of $q$-difference operators:
$$
D_{d}(\xx;q,t)= \sum_{\substack{d_1+\ldots +d_k=d\\ d_i\geq 0}} \prod_{i,j=1}^k \frac{(t x_j/x_i)_{d_j}}{(q x_j/x_i)_{d_j}} \frac{(q x_j/x_i)_{d_j-d_i}}{(t x_j/x_i)_{d_j-d_i}} \prod_{i=1}^k  p_i^{d_i}
$$
where $\xx=(x_1,\ldots, x_k)$ is a set of variables, $q$ and $\hbar$ are parameters, 
$$
(x)_n:=\prod_{i=0}^{n-1}(1-x q^i)
$$ 
is the $q$-Pochammer symbol, and the operator $p_i$ shifts $x_i$ by $q$. We prove that this family of operators acts diagonally on Macdonald polynomials with eigenvalues given as follows:
$$
    D_d(\xx;q,t) P_{\mu}(\xx;q,q/t)=  \frac{(t)_d}{(q)_d} P_{d}(q^{\mu_i}(t/q)^{i-1};q,t) P_{\mu}(\xx;q,q/t)
$$
This result can be reinterpreted as an evaluation formula for a $q$-integral of Selberg type, similar to those studied in \cite{Kaneko} and \cite{OWselb}. As discussed in Section 5.4, the previous formula is equivalent to
\begin{multline*}
\int_{[0,\aa]} P_{\mu}(\xx;q,t) h(\xx) \prod_{i,j=1}^{k} \frac{\varphi\left( tx_j/x_i \right)}{\varphi\left( q x_j/x_i \right)}  \frac{\varphi\left( qx_j/a_i \right)}{\varphi\left( t x_j/a_i \right)}  d_q \xx \\
=P_{\mu}(\aa;q,t) \prod_{i=1}^{k}\frac{\varphi(t q^{\mu_i} (t/q)^{i-1} z)}{\varphi(q^{\mu_i} (t/q)^{i-1} z)}
\end{multline*}
where 
$$
h(\xx):=\exp\left(\frac{1}{\ln(q)} \ln(z) \ln(x_1\ldots x_k) \right), \, \, \, \, \varphi(x):=\prod_{i=0}^{\infty} (1-xq^i)
$$
and the $q$-integral is defined as
$$
\int_{[0,\aa]} g(\xx) d_q \xx := \sum_{d_1,\ldots, d_k \geq 0} g(a_1 q^{d_1}, \ldots, a_k q^{d_k})
$$

Using these properties, we prove that the generating function for the family of operators $D_d(\xx;q,t)$ transforms the vertex function of $X_{\lambda}$ into the vertex function of $X^{!}$. Finally, using the formula Theorem 1 from \cite{dinksmir2}
$$
\textbf{V}_{\lambda}(\zz) = \prod_{\square \in \lambda} \frac{\varphi(\hbar z_{\square})}{\varphi(z_{\square})}
$$
where $\lambda$ is the length $n-k$ partition $(k,k,\ldots,k)$, and $z_{\square}$ is a certain monomial depending on $\square \in \lambda$, we are able to explicitly observe the effect of applying the difference operators to $\textbf{V}_{\lambda}(\zz)$. From this it follows that, after an appropriate normalization and identification of parameters, the vertex functions of $X$ and $X^{!}$ at the last fixed point coincide.

\subsection{Acknowledgements}
We would like to thank Andrey Smirnov for suggesting this project and for his guidance throughout its completion. We would also like to thank Ole Warnaar and Ivan Cherednik for enlightening conversations. In particular, Ole Warnaar brought the paper \cite{NSmac} to our attention. We are grateful to Masatoshi Noumi for sharing the unpublished work \cite{NoumiSano} with us.

\section{Three quiver varieties}
In the course of this paper, the relationship between the vertex functions for three quiver varieties will be important. Though we are mainly interested in the first two ($X=T^*Gr(k,n)$ and its symplectic dual $X^{!}$), the third arises in a natural way when considering the vertex function of $X^{!}$. 
\subsection{The quiver variety $X$}
Fix $k,n \in \matN$. We consider the cotangent bundle to the Grassmannian parameterizing $k$-dimensional subspaces in $\matC^n$. This variety can be described as a Nakajima quiver variety corresponding to the quiver with a single vertex and no edges. The dimension is $k$ and the framing dimension is $n$, and we call the corresponding vector spaces $V$ and $W$, respectively. We choose the stability condition given by the $GL(V)$ character
$$
\theta: g \mapsto \det(g)^{-1}
$$

\begin{figure}[ht]
\centering
\begin{tikzpicture}[roundnode/.style={circle, draw=black, thick, minimum size=8mm},squarednode/.style={rectangle, draw=black, thick, minimum size=8mm}] 

\node[squarednode](F) at (0,-2){$n$};
\node[roundnode](V) at (0,0){$k$};
\draw[thick, ->] (F) -- (V);
\end{tikzpicture}
\caption{The quiver data for the variety $X$.} \label{var}
\end{figure}
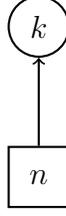

Let $T^*Rep(k,n)$ be the cotangent space of the vector space of framed representations of the quiver with dimenions $k$ and $n$. By definition, $Rep(k,n)=Hom(W,V)$ and so $T^*Rep(k,n)=Hom(W,V)\times Hom(V,W)$. The corresponding Nakajima quiver variety is defined to be the symplectic reduction
$$
X:=T^*Rep(k,n)/\!\!/\!\!/\!\!/_{\!\!\theta} GL(V) := \mu^{-1}(0)^{\theta-ss} / GL(V)
$$
where $\mu: T^*Rep(k,n) \to \mathfrak{gl}(n)^*$ is the moment map for the $GL(V)$ action and $\mu^{-1}(0)^{\theta-ss}$ is the intersection of $\mu^{-1}(0)$ with the $\theta$-semistable points in $T^*Rep(v,w)$. By \cite{GinzburgLectures} Proposition 5.1.5, it follows that $(I,J)\in T^*Rep(v,w)$, is $\theta$ semistable if and only if $I:V\to W$ is injective. The moment map is $\mu(I,J)=I \circ J$. So we see that the quiver variety $X$ is 
$$
X:=\{(I,J)\in T^*Rep(k,n) : \ker I =0, I \circ J =0, \} / GL(V) = T^*Gr(k,n)
$$

\subsection{Fixed points on $X$}
The action of $\bA:=\left(\mathbb{C}^\times\right)^n$ on $W$ induces an action of $\bA$ on $X$, which preserves the symplectic form. Let $\mathbb{C}^\times_\hbar$ act on $X$ by scaling the cotangent directions with character $\hbar^{-1}$. The torus $\mathbb{C}^\times_\hbar$ scales the symplectic form with character $\hbar.$ Let $\bT:=\bA\times \mathbb{C}^\times_h$. We will denote the coordinates on $\bA$ by $\aa=(a_1,\ldots,a_n)$.

The torus $\bT$ fixes the subspaces of $V$ spanned by $k$ coordinate vectors. So there are $\frac{n!}{k!(n-k)!}$ fixed points, naturally indexed by size $k$ subsets of $\{1,\ldots, n\}$.

\subsection{The variety $X^{!}$}
The variety $X^{!}$ dual to $T^*Gr(k,n)$ can be described as a Nakajima quiver variety in the case that $2k \leq n$, see \cite{MirSym1}. For general $n$, such a variety can be described as a bow variety, see \cite{NakBow}. In this paper, we will always assume that $2k\leq n$.

We consider the $A_{n-1}$ quiver with vertices labeled by $1, \ldots, n-1$ and dimension vector $\mathsf{v}=(\mathsf{v}_1,\ldots,\mathsf{v}_{n-1})$ given by
$$\mathsf{v}_i =
\begin{cases}
i & 1 \leq i \leq k-1 \\
k & k \leq i \leq n-k \\
 n-i & n-k+1 \leq i \leq n-1
\end{cases}
$$
and framing dimension $\mathsf{w}=(\mathsf{w}_1,\ldots,\mathsf{w}_{n-1})$ given by $\mathsf{w}_i=\delta_{i,k} + \delta_{i,n-k}$. The corresponding Nakajima quiver variety is defined as the symplectic reduction of $T^*Rep(\mathsf{v},\mathsf{w})$, the cotangent space of framed representations of the $A_{n-1}$ quiver.

\begin{figure}[ht]
\centering
\begin{tikzpicture}[roundnode/.style={circle, draw=black, thick, minimum size=6mm},squarednode/.style={rectangle, draw=black, thick, minimum size=6mm}] 
\node[squarednode](F1) at (6,-2){1};
\node[squarednode](F2) at (9,-2){1};
\node[roundnode](V1) at (1,0){1};
\node[roundnode](V2) at (3,0){2};
\node(V3) at (4.5,0){\ldots};
\node[roundnode](V4) at (6,0){$k$};
\node(V5) at (7.5,0){\ldots};
\node[roundnode](V6) at (9,0){$k$};
\node(V7) at (10.5,0){\ldots};
\node[roundnode](V8) at (12,0){2};
\node[roundnode](V9) at (14,0){1};
\draw[thick, ->] (F1) -- (V4);
\draw[thick, ->] (F2) -- (V6);
\draw[thick, ->] (V1) -- (V2);
\draw[thick, ->] (V2) -- (V3);
\draw[thick, ->] (V3) -- (V4);
\draw[thick, ->] (V4) -- (V5);
\draw[thick, ->] (V5) -- (V6);
\draw[thick, ->] (V6) -- (V7);
\draw[thick, ->] (V7) -- (V8);
\draw[thick, ->] (V8) -- (V9);
\end{tikzpicture}
\caption{The quiver data for the variety $X^{!}$.} \label{dual}
\end{figure}
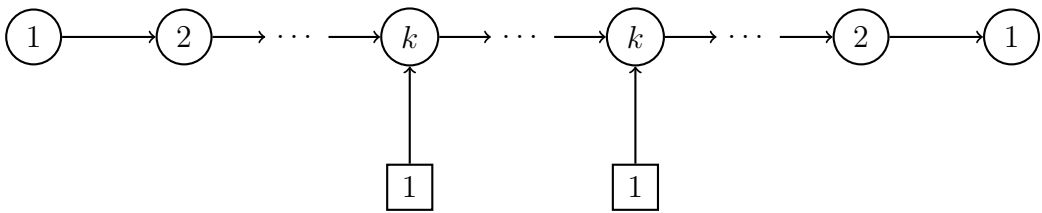

Explicitly,
$$
Rep(\mathsf{v},\mathsf{w})= \bigoplus_{i=1}^{n-2} Hom(V_i,V_{i+1}) \oplus Hom(\matC,V_{k}) \oplus Hom(\matC,V_{n-k})
$$
where $V_i$ is a vector space of dimension $\mathsf{v}_i$ for $1\leq i \leq n-1$.
We use the stability condition given by the $G:=\prod_{i=1}^{n-1} GL(\mathsf{v}_i)$ character
$$
\theta: (g_i) \mapsto \prod_{i=1}^{n-1} \det(g_i)
$$
By definition, the Nakajima quiver variety is the symplectic reduction
$$
X^{!}:=T^*Rep(\mathsf{v},\mathsf{w}) /\!\!/\!\!/\!\!/_{\!\!\theta} G = \mu^{-1}(0)^{\theta-ss} / G
$$
where $\mathfrak{g}=\Lie(G)$, $\mu:T^*Rep(\mathsf{v},\mathsf{w}) \to \mathfrak{g}^* $ is the moment map associated to the $G$ action on $T^*Rep(\mathsf{v},\mathsf{w})$, and $\mu^{-1}(0)^{\theta-ss}$ denotes the intersection of $\mu^{-1}(0)$ with the set of $\theta$-semistable points.

Points in $T^*Rep(\mathsf{v},\mathsf{w})$ are represented by tuples $(X_i,Y_i,I_k,J_k,I_{n-k},J_{n-k})$, where 
$$
X_i: V_i \to V_{i+1}, \, \, \, \, \, Y_i: V_{i+1} \to V_i, \, \, \, \, \, I_{l}:\mathbb{C}\to V_{l}, \, \, \, \, \, J_{l}:V_{l} \to \mathbb{C},\, \, \, \, \,  l \in \{k, n-k\}
$$
By \cite{GinzburgLectures} Proposition 5.1.5, a point is $\theta$-semistable if and only if the image of $I_{k}$ and $I_{n-k}$ generate $\bigoplus_{i=1}^{n-1} V_i$ under the action of all $X_i$ and $Y_i$.

\subsection{Fixed points on $X^{!}$}
The torus $\bA^{!}:=\left(\matC^\times\right)^2$ acts on $X^{!}$ by scaling the framing vector spaces. An additional torus $\matC^\times_{\hbar^{!}}$ acts on $X^{!}$ by scaling the cotangent fiber with character $1/\hbar^{!}$. We let $\bT^{!}:=\bA^{!} \times \mathbb{C}^\times_{\hbar^{!}}$. We will denote the coordinates on $\bA^{!}$ by $(u_1,u_2)$.

A standard argument, see \cite{MirSym1} Section 4, shows that fixed points on the variety $X^{!}$ are indexed by Young diagrams which fit into a $k \times (n-k)$ rectangle.

\subsection{The variety $X_{\lambda}$}
Let $\lambda=(\lambda_1\geq \lambda_2\geq \ldots)$ be a partition. If $l(\lambda)$ is the length of the partition, the Young diagram of the partition is the set of points
$$
\{(x,y)\in \mathbb{Z}^2 : 1\leq x \leq l(\lambda), 1 \leq y \leq \lambda_{x} \}
$$
We do not distinguish between a partition and its Young diagram. If $\square\in \lambda$ has coordinates $(i,j)$, then we define the content of $\square$ to be $c_{\lambda}(\square)=i-j+\lambda_1$. The shift by $\lambda_1$ guarantees that the smallest possible content is 1, which allows for notational agreement with later parts of this paper.

Let $\mathsf{v}=(\mathsf{v}_i)$ where $\mathsf{v}_i=|\{\square \in \lambda : c_{\lambda}(\square)=i\}|$. We define the framing dimension vector to be $\mathsf{w}=(\mathsf{w}_i)$, where $\mathsf{w}_i=\delta_{i,\lambda_1}$. In other words, there is one framing, located at position $\lambda_1$.

Let $X_{\lambda}$ be the corresponding $A_{\infty}$ quiver variety, defined by the stability condition
$$
(g_i) \mapsto \prod_{i} \det(g_i)
$$
where $g_i \in GL(\mathsf{v}_i)$ for $1\leq i \leq \lambda_1+l(\lambda)-1$.

Although the variety $X_{\lambda}$ is geometrically just a point (\cite{dinksmir2} Proposition 1), we will work equivariantly with respect to the torus $\bT_{\lambda}=\mathbb{C}^\times_{\hbar^{!}}$, which acts on the prequotient data by scaling the cotangent directions with character $1/\hbar^{!}$. Thus the tautological bundles over $X_{\lambda}$ carry natural actions of the torus $\bT_{\lambda}$.

\section{Vertex functions for Nakajima quiver varieties}
In this section, we define the main objects of interest. For complete definitions, see \cite{pcmilect} Sections 4-7, \cite{OkBethe}, and \cite{qm}. For various specific examples, see \cite{dinksmir3}, \cite{dinksmir2}, \cite{KorZet}, and \cite{Pushk1}.
\subsection{Equivariant quasimap counts}
Let $X$ be a Nakajima quiver variety from a quiver with vertex set $I$. Let $p \in X^{\bT}$, where $\bT$ is the maximal torus of $Aut(X)$. Associated to $X$ and a choice of degree $\dd\in \mathbb{Z}^{|I|}$, there exists a moduli space of quasimaps 
$$
\qm_p^{\dd}:=\{\text{degree} \, \, d \, \, \text{quasimaps} \, \, f: \mathbb{P}^1 \dashrightarrow X \, \, \text{such that} \, \, f(\infty)=p\} / \cong
$$
where $\cong$ indicates that quasimaps are considered up to isomorphism. It is known that the space $\qm_p^{\dd}$ is a Deligne-Mumford stack of finite type with a perfect obstruction theory, see \cite{qm}. Let $\vss^d$ be the corresponding symmetrized virtual structure sheaf, which is related to the usual virtual structure sheaf provided by the obstruction theory by a twist by a square root of the virtual canonical bundle.

The action of $\bT$ on $X$ induces on action on $\qm_p^{\dd}$. There is an additional action of $\mathbb{C}^\times_q$ by rotation on the domain $\mathbb{P}^1$ of the quasimaps. The moduli spaces $\qm_p^{\dd}$ are not proper. However, the fixed locus $\left(\qm_p^{\dd}\right)^{\bT\times \mathbb{C}^{\times}_q}$ is. Thus the equivariant Euler characteristic is well-defined:
$$
\chi(\vss^{\dd}) \in \mathbb{Q}(\aa,\hbar,q)
$$
where $\hbar$ denotes the character of the symplectic form and $\aa$ denotes the coordinates on the torus $\bA:=\ker(\hbar)\subset \bT$ preserving the symplectic form.

Quasimap spaces come equipped with natural evaluation maps $\text{ev}_b$ for $b\in \mathbb{P}^1$, taking values in the quotient stack $[\mu^{-1}(0)/G]$, where $\mu$ is the moment map and $G$ is the gauge group in the definition of the quiver variety. Given a class $\tau\in K_{\bT}(X)$, there is an associated class $\tau_{\text{stack}}\in K_{\bT}([\mu^{-1}(0)/G])$. We can pullback this class under $\text{ev}_0$ to a class on $\qm_p^{\dd}$, which we will also denote by $\tau$.

\begin{Definition}
The vertex function of $X$ at $p$ with descendant $\tau$ is defined as the generating function of the $\tau$-twisted equivariant Euler characteristics of $\qm^{\dd}_p$:
$$
\textbf{V}_p(\zz) \left\langle \tau \right\rangle := \sum_{\dd} \chi(\vss^{\dd} \otimes \tau) \zz^\dd
$$
where the sum is taken over the cone of degrees such that the space $\qm_p^{\dd}$ is nonempty (see \cite{pcmilect} Section 7.2). Here $\zz^\dd=z_1^{d_1}\ldots z_{|I|}^{d_{|I|}}$.
\end{Definition}
The variables $z_i$ are known as the K\"ahler parameters. The vertex function also depends on the equivariant parameters, which we sometimes write as an argument below.

\subsection{Integral form of vertex functions}
It is known that the vertex functions for quiver varieties can be represented as integrals, see \cite{OkBethe}. Let $X$ be a Nakajima quiver variety with dimension and framing dimension vectors $\mathsf{v}$ and $\mathsf{w}$, and associated vector spaces $V_i$ and $W_i$. The index $i$ takes values in $I$, the vertex set of the quiver. Let $G=\prod_{i\in I} GL(V_i)$.

For a character $x_1+\ldots+x_m\in K_{\bT}(X)$, we define 
$$
\Phi(x_1+\ldots +x_m):= \varphi(x_1)\ldots \varphi(x_m), \, \, \, \, \, \, \, \, \varphi(x):=\prod_{i=0}^{\infty} (1-x q^i)
$$
We extend $\Phi$ by linearity to polynomials with negative coefficients. 

Let $\mathcal{P}$ be the bundle over $X$ associated to the virtual $G$-module
$$
\bigoplus_{i \to j} Hom(V_i,V_j) + \bigoplus_{i \in I} Hom(W_i,V_i) - \bigoplus_{i \in I} Hom(V_i,V_i)
$$
where $i\to j$ denotes the sum over the arrows of the quiver. The quiver variety is equipped with a collection of tautological bundles $\mathcal{V}_i$, $i \in I$. Let $x_{i,j}$, $1\leq j \leq \mathsf{v}_i$ be the Chern roots of the bundle $\mathcal{V}_i$. It is known that the tautological bundles generate the $K$-theory of $X$, see \cite{kirv}.

As a formal expression in the Chern roots and K\"ahler parameters, we define the following:
$$
f(\xx,\zz):= \exp\left(\frac{1}{\ln(q)}\sum_{i \in I} \sum_{j=1}^{\mathsf{v}_i} \ln(z_i) \ln(x_{i,j})\right) = \exp\left(\frac{1}{\ln(q)}\sum_{i \in I} \ln(z_i) \ln\left(\mathcal{L}_i\right) \right)
$$
where $\mathcal{L}_i=\det\left(\mathcal{V}_i\right)$. Shifting a Chern root by $q$ gives the following transformation property:
$$
f(x_{1,1},\ldots,qx_{i,j},\ldots,\zz) = z_i f(\xx,\zz)
$$
For a function $g(\xx)$ of the Chern roots of the tautological bundles, symmetric in the variables $x_{i,1}, \ldots, x_{i,\mathsf{v}_i}$ for each $i$ and a $\bT$-fixed point $p$, we define the $q$-integral as
$$
\int_0^p g(\xx) d_q \xx := \sum_{d_{i,j}=0}^{\infty} g(q^{\dd} \xx_p) 
$$
where each $d_{i,j}$ is summed from 0 to $\infty$ and $q^{\dd} \xx_p$ denotes the substitution of the weights of the tautological bundles $\mathcal{V}_i$ at $p$ in place of the Chern roots, shifted by $q^{d_{i,j}}$. Then it is known (see \cite{OkBethe}) that the vertex function of $X$ at $p$ with descendant $\tau$ is equal to 
\begin{multline}\label{qint}
  \textbf{V}_p(\zz)\left\langle \tau \right \rangle \\
  =\Phi((q-\hbar)\mathcal{P}_p)^{-1} f(\xx_p,\zz)^{-1}  \int_0^p \Phi\left((q-\hbar)\mathcal{P}(\xx)\right) f(\xx,\zz) \tau(\xx) d_q\xx
\end{multline}
where $\mathcal{P}_p$ is the $\bT$-character of $\mathcal{P}$ at $p$, and $\mathcal{P}(\xx)$ and $\tau(\xx)$ denote the expression of the classes $\mathcal{P}$ and $\tau$ in terms of the Chern roots $\xx$. Since all expressions in the integral are symmetric functions of the Chern roots, the substitution $\xx_p$ is well-defined.

From the transformation properties of $f$ and $\varphi$, it is clear that the summation on the right hand side of (\ref{qint}) is a power series in the K\"ahler parameters with coefficients in $\mathbb{Q}(\aa,\hbar,q)$. The $q$-integral formula for the vertex function arises from computing the $\tau$-twisted equivariant Euler characteristics $\chi(\vss^{\dd}\otimes \tau)$ by equivariant localization.

\section{Symplectic duality of $X$ and $X^{!}$}
In this section, we investigate the relationship between the vertex functions of $X$ and $X^{!}$ at a particular fixed point. In what follows, the restriction of the vertex functions of $X$ and $X^{!}$ to fixed points $p$ and $p^{!}$ will be denoted by $\textbf{V}_p(\aa,z)$ and $\textbf{V}^{!}_{p^{!}}(u,\zz)$, respectively. The vertex function of $X_{\lambda}$ will be denoted by $\textbf{V}_{\lambda}(\zz)$.

\subsection{Bijection on fixed points}
As explained above, the $\bT^{!}$-fixed points on $X^{!}$ are naturally indexed by Young diagrams that fit into a $k\times (n-k)$ rectangle. We can consider the path traced out by the border of the diagram. This path is completely determined by specifying which of the $n$ segments of the path move vertically. In order for the Young diagram to fit into a $k \times (n-k)$ rectangle, there must be exactly $k$ such locations, which corresponds to a size $k$ subset of $\{1,\ldots, n\}$. This gives a natural bijection between fixed points on $X^{!}$ and fixed points on $X$. See Figure \ref{bij} for an example.

\begin{figure}[ht]
\centering
\begin{tikzpicture}[roundnode/.style={circle, draw=black, thick, minimum size=8mm},squarednode/.style={rectangle, draw=black, thick, minimum size=8mm}] 
\draw[thick, dotted] (0,0) -- (3,0);
\draw[thick, dotted] (0,1) -- (4,1);
\draw[thick, dotted] (0,2) -- (1,2);
\draw[thick,dotted] (3,2)--(4,2);
\draw[thick, dotted] (1,3) -- (4,3);
\draw[thick, dotted] (0,0) -- (0,3);
\draw[thick, dotted] (1,0) -- (1,2);
\draw[thick, dotted] (2,0) -- (2,3);
\draw[thick, dotted] (3,2) -- (3,3);
\draw[thick, dotted] (4,0) -- (4,3);
\draw[thick, -] (0,3)--(1,3)--(1,2)--(3,2)--(3,0)--(4,0);
\end{tikzpicture}
\caption{A fixed point where $k=3$ and $n=7$. The fixed point corresponds to the subset $\{2,5,6\}\subset \{1,2,3,4,5,6,7\}$} \label{bij}
\end{figure}

In particular, we consider the fixed point $p$ on $X$ given by the subset $\{n-k+1,\ldots,n\}$, or equivalently, the fixed point $p^{!}$ on $X^{!}$ given by the length $n-k$ partition $\lambda=(k,k,\ldots,k)$. For the remainder of this paper, we will always assume that $p$ and $p^{!}$ denote these fixed points and we will always use $\lambda$ to denote the partition $(k,k,\ldots,k)$.

\subsection{Exchange of K\"ahler and equivariant parameters}
Let $\widetilde{\bA}$ be the cokernel of diagonal inclusion $\mathbb{C}^\times \to \bA$ and let $\widetilde{\bA}^{!}$ be the cokernel of the diagonal inclusion $\mathbb{C}^\times \to \bA^{!}$.

Define the map $\kappa: \bK^{!}\times \widetilde{\bA}^{!} \times \mathbb{C}^\times_{\hbar^{!}} \times \mathbb{C}^\times_{q} \to \bK\times \widetilde{\bA} \times \mathbb{C}^\times_{\hbar} \times \mathbb{C}^\times_{q} $ by
\begin{align*}
    &z_i\mapsto
\begin{cases}
 \frac{1}{\hbar} \frac{a_{i+1}}{a_{i}} & 1 \leq i <k \\
\frac{a_{i+1}}{a_{i}}  & k \leq i < n-k \\
\hbar \frac{a_{i+1}}{a_{i}} & n-k \leq i \leq n-1
\end{cases} \\
&\hbar^! \mapsto \frac{q}{\hbar} \\
& \frac{u_1}{u_2} \mapsto z \left(\frac{\hbar}{q}\right)^k \\
& q \mapsto q
\end{align*}

This map is an isomporphism of tori, and gives an induced map $\mathbb{C}(q,\hbar^{!},u)[[\zz]] \to \mathbb{C}(q,\hbar,z)[[\aa]]$ where $u=u_1/u_2$, which we also denote by $\kappa$. 

\subsection{Vertex Function of $X$}
At the fixed point $p$, the $\bT$-character of the tautological bundle $\mathcal{V}$ on $X$ is 
$$
\mathcal{V}=a_{n-k+1}+\ldots +a_n \in K_{\bT}(p)
$$
Using the integral representation, it is straightforward to see that the vertex function of $T^*Gr(k,n)$ at the fixed point $p$ is equal to
\begin{align}
\textbf{V}_p(\aa,z) &= 
\sum_{d_1,\ldots,d_k=0}^{\infty} \prod_{i=1}^{n} \prod_{j=n-k+1}^n \frac{\left(\hbar a_j/a_i\right)_{d_j}}{\left(q a_j/a_i\right)_{d_j}} \prod_{i,j=n-k+1}^n \frac{\left(q a_j/a_i\right)_{d_j-d_i}}{\left(\hbar a_j/a_i\right)_{d_j-d_i}} z^{d_1+\ldots d_k}
\end{align}

\subsection{Vertex function of $X^{!}$}
At the fixed point $p^{!}$, the character of the tautological bundle $\mathcal{V}_i$ is 
$$
\mathcal{V}_i = \begin{cases}
\sum_{j=1}^{\mathsf{v}_i}  u_1 \hbar^{j-i-1} & 1\leq i< k \\
\sum_{j=1}^{\mathsf{v}_i} u_1 \hbar^{j-1} & k \leq i \leq n-1
\end{cases} 
$$
Using the integral representation, it is straightforward to see that the vertex function of $X^{!}$ at the fixed point $p^{!}$ is equal to
\begin{align} \nonumber
\textbf{V}^{!}_{p^{!}}(u,\zz) &= \sum_{d_{i,j}=0}^{\infty} \prod_{i=1}^{\mathsf{v}_k} \frac{(\hbar^{i})_{d_{k,i}}}{(q\hbar^{i-1})_{d_{k,i}}} \prod_{i=1}^{\mathsf{v}_{n-k}} \frac{(\hbar^{i} u)_{d_{n-k,i}}}{(q \hbar^{i-1} u)_{d_{n-k,i}}} \prod_{i=1}^{k-1} \prod_{j=1}^{\mathsf{v}_{i}} \prod_{l=1}^{\mathsf{v}_{i+1}} \frac{(\hbar^{l-j})_{d_{i+1,l}-d_{i,j}}}{(q \hbar^{l-j-1})_{d_{i+1,l}-d_{i,j}}} \\
&  \prod_{i=k}^{n-1} \prod_{j=1}^{\mathsf{v}_{i}} \prod_{l=1}^{\mathsf{v}_{i+1}} \frac{(\hbar^{l-j+1})_{d_{i+1,l}-d_{i,j}}}{(q \hbar^{l-j})_{d_{i+1,l}-d_{i,j}}} \prod_{i=1}^{n-1} \prod_{i,l=1}^{\mathsf{v}_i} \frac{(q \hbar^{l-j})_{d_{i,l}-d_{i,j}}}{(\hbar^{l-j+1})_{d_{i,l}-d_{i,j}}} \prod_{i=1}^{n-1} \prod_{j=1}^{\mathsf{v}_i} z_i^{d_{i,j}}
\end{align}
where $u=u_1/u_2$ and the sum is taken over $d_{i,j}$ for $1\leq i \leq n-1$, $1 \leq j \leq \mathsf{v}_i$ from 0 to $\infty$.

\begin{Note}
In general, the degrees $d_{i,j}$ must lie inside of a certain cone. But in the case of the vertex function above, the specific form of the coefficients gives 0 whenever the degrees lie outside this cone. Hence for uniformity, we prefer to think of each $d_{i,j}$ as running from 0 to $\infty$.
\end{Note}

\subsection{Vertex function of $X_{\lambda}$}

Since the same quiver is used to define both $X^{!}$ and $X_{\lambda}$, we can canonically identify the K\"ahler parameters for these two varieties. We also use the same notation $\hbar^{!}$ for the character of the symplectic form for these two varieties. Thus $\textbf{V}_{\lambda}(\zz)$ is a power series in $z_i$ for $1\leq i \leq n-1$, with coefficients in $\mathbb{Q}(q,\hbar^!)$. Explicitly, we have
\begin{align}\label{vpt} \nonumber
\textbf{V}_{\lambda}(\zz) &= \sum_{d_{i,j}=0}^{\infty} \prod_{i=1}^{\mathsf{v}_k} \frac{(\hbar^{i})_{d_k}}{(q\hbar^{i-1})_{d_k}} \prod_{i=1}^{k-1} \prod_{j=1}^{\mathsf{v}_{i}} \prod_{l=1}^{\mathsf{v}_{i+1}} \frac{(\hbar^{l-j})_{d_{i+1,l}-d_{i,j}}}{(q \hbar^{l-j-1})_{d_{i+1,l}-d_{i,j}}} \\
&  \prod_{i=k}^{n-1} \prod_{j=1}^{\mathsf{v}_{i}} \prod_{l=1}^{\mathsf{v}_{i+1}} \frac{(\hbar^{l-j+1})_{d_{i+1,l}-d_{i,j}}}{(q \hbar^{l-j})_{d_{i+1,l}-d_{i,j}}} \prod_{i=1}^{n-1} \prod_{i,l=1}^{\mathsf{v}_i} \frac{(q \hbar^{l-j})_{d_{i,l}-d_{i,j}}}{(\hbar^{l-j+1})_{d_{i,l}-d_{i,j}}} \prod_{i=1}^{n-1}z_i^{d_i}
\end{align}
where the sum is taken over $d_{i,j}$ for $1\leq i \leq n-1$, $1 \leq j \leq \mathsf{v}_i$ from 0 to $\infty$ and $\mathsf{v}_i=|\{\square \in \lambda : c_{\lambda}(\square)=i\}|$ as before. 

Let 
\begin{equation*}
   z_{\square} := \prod_{\square' \in H_{\lambda}(\square)} \widehat{z}_{c(\square')}
\end{equation*}
where the shifted K\"ahler parameters $\widehat{z}_i$ are
\begin{equation*}
    \widehat{z}_i:=\left(\frac{\hbar}{q}\right)^{\sigma_{\lambda}(i)} z_i \ \ \ \text{where} \ \ \   \sigma_{\lambda}(i):= \begin{cases} 
      \textsf{v}_{i-1}-\textsf{v}_{i} & \text{if} \ \ i \neq 0 \\
 \textsf{v}_{i-1}-\textsf{v}_{i}+1 & \text{if} \ \ i = 0
   \end{cases} 
\end{equation*}
and $H_{\lambda}(\square)$ denotes the set of boxes in the hook based at $\square$ in $\lambda$. If $\square$ has coordinates $(i,j)$ in the Young diagram for $\lambda$, then 
$$
H_{\lambda}(\square)=\{(i,m) \in \lambda : m\geq j\} \cup \{(m,j) \in \lambda : m\geq i \}
$$

\begin{Proposition}[\cite{dinksmir2} Theorem 1]\label{verpoint} For $|q|<1$, $\textbf{V}_{\lambda}(\zz)$ is the power series expansion of the function
$$
\textbf{V}_{\lambda}(\zz) = \prod_{\square \in \lambda} \frac{\varphi(\hbar z_{\square})}{\varphi(z_{\square})}
$$
holomorphic in the polydisk $|z_{\square}|<1$.
\end{Proposition}

Proposition \ref{verpoint} holds for quiver varieties associated to general partitions. But for our specific case of $\lambda=(k,k,\ldots,k)$, we obtain
\begin{Proposition}\label{prform}
$$
\kappa\left( \textbf{V}_{\lambda}(\zz)\right) = \prod_{j=n-k+1}^{n} \prod_{i=1}^{n-k} \frac{\varphi(q a_j/a_i) }{\varphi(\hbar a_j/a_i)}
$$
\end{Proposition}
\begin{proof}
This follows from a direct calculation using Proposition \ref{verpoint} and the definitions of $z_{\square}$ and $\kappa$.
\end{proof}

\subsection{Main theorem, coincidence of normalized vertex functions}
The main result in this paper is the coincidence of the vertex functions for $X$ and $X^{!}$, after normalization by a simple prefactor. Let
$$
\widetilde{\textbf{V}}_{p^{!}}^{!}(u,\zz)= \textbf{V}_{\lambda}(\zz)^{-1} \textbf{V}_{p^{!}}^{!} (u,\zz)
$$
and 
$$
\widetilde{\textbf{V}}_{p}(\aa,z) = \prod_{i=1}^{k} \frac{\varphi( (\hbar/q)^{i-1} z)}{\varphi(\hbar (\hbar/q)^{i-1} z)}\textbf{V}_p(\aa,z)
$$
Using the $q$-binomial theorem and Proposition \ref{prform}, $\widetilde{\textbf{V}}_{p^{!}}^{!}(u,\zz)$ can be expanded as a power series in the K\"ahler parameters. Applying the map $\kappa$, we obtain an element of $\mathbb{C}(q,\hbar,z)[[\aa]]$. Each term involving $z$ appears in the form 
$$
\frac{1-wz}{1-w'z}
$$
where $w$ and $w'$ are monomials in $q$ and $\hbar$. Thus, we can expand each term as a power series in $z$ and obtain an element of $\mathbb{C}(q,\hbar)[[\aa,z]]$. Collecting powers of $z$ gives
$$
\kappa \left( \widetilde{\textbf{V}}^{!}_{p^{!}}(u,\zz)\right) = \sum_{d=0}^{\infty} c_{d}(\aa,q,\hbar) z^d
$$
where $c_{d}(\aa,q,\hbar)\in \mathbb{C}(q,\hbar)[[\aa]]$.

Similarly, applying the $q$-binomial theorem to the prefactor, we identify $\widetilde{\textbf{V}}_p(\aa,z)$ as an element of $\mathbb{C}(q,\hbar,\aa)[[z]]$.

\begin{Theorem}\label{main}
In the notation above, each $c_{d}(\aa,q,\hbar)$ is the Taylor series expansion of a rational function of $\aa$ holomorphic in a punctured neighborhood of $0$. Furthermore, as elements of $\mathbb{C}(q,\hbar,\aa)[[z]]$, we have the equality:
$$
\widetilde{\textbf{V}}_p(\aa,z) =\kappa\left( \widetilde{\textbf{V}}^{!}_{p^{!}}(u,\zz)  \right)
$$
\end{Theorem}
This theorem will be proven below.

\begin{Note}
Like the normalizing factor for the vertex function of $X^{!}$, the normalizing factor for the vertex function of $X$ is the vertex function for a simple quiver variety related to $X$. In this case, it is the vertex function for the zero dimensional quiver variety $T^*Gr(k,k)$, obtained from the one-vertex quiver with dimension and framing dimension both equal to $k$.
\end{Note}

\section{Difference operators and descendant insertions}

In this section, we introduce some machinery and prove Theorem \ref{main}.

\subsection{Macdonald polynomials}
Let $\mathcal{F}=\mathbb{C}[x_1,\ldots,x_k]^{S_k}\otimes \mathbb{C}(q,t)$ be the ring of symmetric polyomials in $x_1,\ldots,x_k$. Following \cite{mac}, we define an inner product on $\mathcal{F}$ by
\begin{equation}
    \langle p_{\lambda}, p_{\mu} \rangle := \delta_{\lambda,\mu} \prod_{n \geq 1} n^{m_n} m_n! \prod_{i=1}^{l(\lambda)} \frac{1-q^n}{1-t^n} \ \ \  \text{where}  \ \ \  m_n=|\{k \mid \lambda_k=n \}|
\end{equation}
where
$$
p_{\lambda} = \prod_{i=1}^{l(\lambda)} p_{\lambda_i} \ \ \ \text{and} \ \ \ p_i=\sum_{j=1}^{k} x_{j}^i
$$
The Macdonald polynomials $P_{\mu}(\xx;q,t)$, where $\mu$ is a partition of length at most $k$, are the unique basis of $\mathcal{F}$ defined by the following two properties
$$
\lambda \neq \mu \implies \langle P_{\mu}(\xx;q,t), P_{\lambda}(\xx;q,t)\rangle = 0
$$
$$
P_{\lambda}(\xx;q,t) = \sum_{\mu \leq \lambda} u_{\lambda\mu} m_{\mu}(\xx), \ \ \ u_{\lambda\lambda}=1,  \ \ u_{\lambda,\mu} \in \matQ(q,h)
$$
where $m_\mu(\xx)$ is the monomial symmetric function corresponding to $\mu$
 and
$$
\mu \leq \lambda \iff \mu_1 + \ldots + \mu_i \leq \lambda_1+\ldots +\lambda_i, \ \ \forall i \geq 0
$$

\subsection{Vertex function for $X^{!}$ as a descendant insertion}
As remarked above, $\textbf{V}^!_{p^{!}}(u,\zz)$ is a power series in $\zz$, with coefficients in $\mathbb{C}(q,\hbar^{!},u)$. The only terms involving $u=u_1/u_2$ appear as 
$$
\frac{1-wu}{1-w'u}
$$
for Laurent monomials $w$ and $w'$ in $q$ and $\hbar^{!}$. Hence we can expand each of these as a power series in $u$, which identifies $\textbf{V}^!_{p^{!}}(u,\zz)$ as an element of $\mathbb{C}(q,\hbar^{!})[[\zz,u]]$. 

For $\tau\in K_{\bT_{\lambda}}(X_{\lambda})$, we write the descendant insertion of $\tau$ into the vertex function as $\textbf{V}_{\lambda}(\zz)\langle \tau \rangle$. We have the following 

\begin{Proposition}\label{reducetopoint} Let $\xx=(x_1, \ldots, x_k)$ be the Chern roots of the tautological bundle $\mathcal{V}_{n-k}$ on $X_{\lambda}$. As elements of $\mathbb{C}(q,\hbar^{!})[[\zz,u]]$, we have
\begin{equation}
\textbf{V}^{!}_{p^{!}}(u,\zz) = \prod_{i=1}^k \frac{\varphi(u  {\hbar^{!}}^{i})}{\varphi(u q {\hbar^{!}}^{i-1})} \textbf{V}_{\lambda}(\zz) \left\langle \sum_{d=0}^{\infty} \frac{\left(q/\hbar^{!}\right)_d}{(q)_d} P_{(d)}\left(\xx; q, q/\hbar^{!}\right) (\hbar^{!} u)^d  \right\rangle
\end{equation}
where $u=u_1/u_2$ and $P_{(d)}$ is the Macdonald polynomial for the length one partition $(d)$.
\end{Proposition}
\begin{proof}
From the explicit form (\ref{vpt}) of $\textbf{V}_{\lambda}(\zz)$, we see that $$
\textbf{V}^{!}_{p^{!}}(u,\zz)= \prod_{i=1}^k \frac{\varphi(u  {\hbar^{!}}^{i})}{\varphi(u q {\hbar^{!}}^{i-1})} \textbf{V}_{\lambda}(\zz)\left\langle \prod_{i=1}^{k} \frac{\varphi(u q x_i)}{\varphi(u \hbar^{!} x_i)} \right\rangle
$$ 
Recall the identity from \cite{mac} Chapter 6 Section 2:
$$
\prod_{i=1}^{k} \frac{\varphi(y t x_i)}{\varphi(y x_i)} = \sum_{d=0}^{\infty} \frac{(t)_d}{(q)_d} P_{(d)}(\xx;q,t) y^d
$$
Substituting $t=q/\hbar^{!}$ and $y=\hbar^{!} u$ gives the result.

\end{proof}

\subsection{A family of difference operators}

\begin{Definition}\label{dop} Let $\xx=(x_1,\ldots,x_k)$ be a set of variables. Define the following difference operator on $\mathbb{C}(q,t,x_1,\ldots,x_k)$
\begin{equation}\label{dopeq}
 D_{d}(\xx;q,t)=\sum_{\substack{d_1+\ldots+d_k=d\\d_i\geq0}} \prod_{i,j=1}^k \frac{(t x_j/x_i)_{d_j}}{(q x_j/x_i)_{d_j}} \frac{(q x_j/x_i)_{d_j-d_i}}{(t x_j/x_i)_{d_j-d_i}} \prod_{i=1}^k  p_i^{d_i}
\end{equation}
where $p_i:x_j \mapsto q^{\delta_{i,j}} x_j$.
\end{Definition}

The motivation for introducing $D_{d}(\xx;q,t)$ comes from the following theorem.
\begin{Theorem}\label{insertion}
Let $\xx=(x_1,\ldots,x_k)$ denote the Chern roots of the tautological bundle $\mathcal{V}_{n-k}$ on $X_{\lambda}$ and let $\aa=(a_{n-k+1},\ldots,a_{n})$. Then
$$
D_d(\aa;q,\hbar) \kappa\left( \textbf{V}_{\lambda}(\zz)\right) = \kappa\left(\textbf{V}_{\lambda}(\zz)\left\langle \frac{\left(q/\hbar^{!}\right)_d}{(q)_d} P_{(d)}\left(\xx;q,q/\hbar^{!}\right) {\hbar^{!}}^{d(1-k)} \right\rangle \right)
$$ 
\end{Theorem}
This will be proven below. 
\begin{Note}
Theorem \ref{insertion} is similar to the main result of \cite{dinksmir3}. There, we studied descendant insertions into $\textbf{V}_{\lambda}(\zz)$ in the basis of elementary symmetric functions in the Chern roots of the tautological bundles on $X_{\lambda}$. In that case, the techniques are similar, and instead of $D_{d}(\aa;q,\hbar)$, the classical Macdonald difference operators appear.
\end{Note}

\subsection{Spectrum of $D_{d}(\xx;q,t)$}

\begin{Theorem}\label{diagonal}
The operators $D_{d}(\xx;q,t)$ for $d\in \matN$ are pairwise commutative. In addition, they act diagonally in the basis of Macdonald polynomials $P_{\mu}(\xx;q,q/t)$ as
\begin{equation}
    D_d(\xx;q,t) P_{\mu}(\xx;q,q/t)=  \frac{(t)_d}{(q)_d} P_{d}(q^{\mu_i}(t/q)^{i-1};q,t) P_{\mu}(\xx;q,q/t)
\end{equation}

\end{Theorem}

\begin{Lemma}
\begin{align*}
    &D_d(\xx;q,t) \\
=&\sum_{d_1+\ldots+d_k=d} \prod_{i,j=1}^{k} \frac{\left(t x_j/x_i\right)_{d_j}}{\left(q x_j/x_i\right)_{d_j}} \prod_{1\leq i < j \leq k} \frac{1-q^{d_j-d_i}x_j/x_i}{1-x_j/x_i} \frac{(qx_j/tx_i)_{d_j-d_i}}{(tx_j/x_i)_{d_j-d_i}} \left(\frac{t}{q} \right)^{d_j-d_i}\prod_{i=1}^k p_i^{d_i}
\end{align*}
\end{Lemma}
\begin{proof}
This follows by elementary manipulations with the $q$-Pochammer symbols in (\ref{dopeq}).
\end{proof}

Define the generating function 
$$
D(z)=\sum_{d=0}^{\infty} z^d D_d(\xx;q,t)
$$
Then 
\begin{Theorem}\label{diagonal2}
Theorem \ref{diagonal} is equivalent to 
\begin{equation}
    D(z) P_{\mu}(\xx;q,q/t)= P_{\mu}(\xx;q,q/t) \prod_{i=1}^{k} \frac{\varphi(t q^{\mu_i} (t/q)^{i-1} z)}{\varphi(q^{\mu_i} (t/q)^{i-1} z)}
\end{equation}
\end{Theorem}
\begin{proof}
This following by expanding the product in the right hand side of the Theorem as a power series in $z$ and equating the coefficients of $z$.
\end{proof}

Let $p(\xx;\ss;q,t)$ be the function defined in formula (1.11) of \cite{NSmac}. It depends on two sets of variables $\xx=(x_1,\ldots,x_k)$ and $\ss=(s_1,\ldots,s_k)$. It has the property that
\begin{equation}\label{spec}
x^{\mu} p(\xx;q^{\mu} t^{\delta};q,t)=P_{\mu}(\xx;q,t)
\end{equation}
where $\delta=(k-1,\ldots,0)$ and $q^{\mu} t^{\delta}$ stands for the specialization $s_i=q^{\mu_i} t^{\delta_i}$. Let
\begin{align*}
    \psi(\xx;\ss;q,t) &= \prod_{1\leq i < j \leq k} \frac{\varphi\left(qx_j/tx_i\right)}{\varphi\left(qx_j/x_i\right)} p(\xx;\ss;q,t)
\end{align*}
and
$$
e(\xx;\ss)= \prod_{i=1}^{k} \frac{\vartheta(x_i t^{k-i}) \vartheta(s_i t^{k-1})}{\vartheta(s_i x_i)}
$$
where $\vartheta(x)=\varphi(x)\varphi(q/x)\varphi(q)$. 

The relevant property of $e(\xx;\ss)$ is given by the following:
\begin{Lemma}
The function $e(\xx;\ss)$ transforms under the action of $p_i$ as:
$$
p_ie(\xx;\ss)=e(\xx;\ss) s_i t^{i-k}
$$
\end{Lemma}
\begin{proof}
This follows from direct computation.
\end{proof}

\begin{Lemma}[\cite{NSmac} Theorem 1.4]
The function $\psi(\xx;\ss;q,t)$ satisfies the following identity:
$$
\psi(\xx;\ss;q,t)=\psi(\xx;\ss;q,q/t)
$$
\end{Lemma}

We also consider the normalized function:
$$
f(\xx;\ss;q,t)=e(\xx,\ss) p(\xx;\ss;q,t)
$$

Recall from \cite{NoumiSano} and \cite{NSmac}, the family of $q$-difference operators
$$
N_{d}(\xx;q,t)= \sum_{d_1+\ldots+d_k=d} \prod_{i,j=1}^{k} \frac{\left(t x_j/x_i\right)_{d_j}}{\left(q x_j/x_i\right)_{d_j}} \prod_{1\leq i < j \leq k} \frac{q^{d_j}x_j-q^{d_i}x_i}{x_j-x_i} \prod_{i=1}^k p_i^{d_i}
$$
These operators are known as Noumi's $q$-difference operators, or the Macdonald operators of row type. Define the generating function
\begin{align*}
    N(z) &=\sum_{d=0}^{\infty} z^d N_{d}(\xx;q,t)
\end{align*}

It is known that
\begin{Proposition}[\cite{NSmac} Formula (5.7)]\label{Noumi}
$$
N(z)f(\xx;\ss;q,t) = f(\xx;\ss;q,t) \prod_{i=1}^k \frac{\varphi(t s_i z)}{\varphi(s_i z)}
$$
\end{Proposition}

\begin{proof}[Proof of Theorem \ref{diagonal2}]
Using the definitions and lemmas given above, we first  calculate
\begin{align*}
    p(\xx;\ss;q,t) &=\prod_{1\leq i<j\leq k} \frac{\varphi(tx_j/x_i)}{\varphi(qx_j/tx_i)} \prod_{1\leq i<j\leq k} \frac{\varphi(qx_j/x_i)}{\varphi(tx_j/x_i)} \psi(\xx;\ss;q,q/t) \\
    &= \Delta(\xx) p(\xx;\ss;q,q/t)
\end{align*}
where 
$$
\Delta(\xx)=\prod_{1\leq i < j \leq k} \frac{\varphi(tx_j/x_i)}{\varphi(qx_j/tx_i)}
$$
Then
\begin{align}\label{ftop}
   f(\xx;\ss;q,t) = \frac{e(\xx;\ss)}{x^{\mu}} \Delta(\xx) x^{\mu} p(\xx;\ss;q,q/t)
\end{align}
where $x^{\mu}=\prod_{i=1}^{k} x_i^{\mu_i}$. Substituting (\ref{ftop}), specializing $\ss=q^{\mu} \left(q/t\right)^{\delta}$, and replacing $z$ by $z (t/q)^{k-1}$ transforms Proposition \ref{Noumi} to 

\begin{align}\label{eq} \nonumber
    N\left(z \left(t/q\right)^{k-1} \right) & e(\xx;\ss) x^{-\mu} \Delta(\xx) P_{\mu}(\xx;q,q/t) \\ &= e(\xx;\ss) x^{-\mu} \Delta(\xx) P_{\mu}(\xx;q,q/t) \prod_{i=1}^{k} \frac{\varphi(t q^{\mu_i} (t/q)^{i-1} z)}{\varphi( q^{\mu_i} (t/q)^{i-1} z)}
\end{align}
where we have suppressed writing the specialization of $\ss$, which we will continue to do below. Rewriting this gives
\begin{multline*}
    \Delta(\xx)^{-1} x^{\mu} e(\xx;\ss)^{-1}  N\left(z (t/q)^{k-1}\right) e(\xx;\ss)x^{-\mu} \Delta(\xx) P_{\mu}(\xx;q,q/t) \\
    = P_{\mu}(\xx,q,q/t) \prod_{i=1}^{k} \frac{\varphi(t q^{\mu_i} (t/q)^{i-1} z)}{\varphi( q^{\mu_i} (t/q)^{i-1} z)}
\end{multline*}
Using the transformation property of $e(\xx;\ss)$ and $x^{\mu}$, we calculate
\begin{align*}
  x^{\mu} e(\xx;\ss)^{-1} & N_{d}(\xx;q,t) e(\xx;\ss) x^{-\mu} \\
  &= \sum_{d_1+\ldots+d_k=d} \prod_{i,j=1}^{k} \frac{\left(t x_j/x_i\right)_{d_j}}{\left(q x_j/x_i\right)_{d_j}} \prod_{1\leq i < j \leq k} \frac{q^{d_j}x_j-q^{d_i}x_i}{x_j-x_i} \prod_{i=1}^k t^{d_i(i-k)} (q/t)^{d_i(k-i)} p_i^{d_i}
\end{align*}
Doing some more elementary manipulations gives 
\begin{multline*}
 x^{\mu} e(\xx;\ss)^{-1}  N_{d}(\xx;q,t) e(\xx;\ss) x^{-\mu} \\
 =  \sum_{d_1+\ldots+d_k=d} \prod_{i,j=1}^{k} \frac{\left(t x_j/x_i\right)_{d_j}}{\left(q x_j/x_i\right)_{d_j}} \prod_{1\leq i < j \leq k} \frac{1-q^{d_j-d_i}x_j/x_i}{1-x_j/x_i}\left(\frac{t}{q}\right)^{d_j-d_i} \prod_{i=1}^k (q/t)^{d_i(k-1)} p_i^{d_i} 
\end{multline*}
It is now straightforward to see that conjugating this by $\Delta(\xx)$ gives
\begin{align*}
    \Delta(\xx)^{-1}  x^{\mu} e(\xx;\ss)^{-1} & N_{d}(\xx;q,t) e(\xx;\ss) x^{-\mu} \Delta(\xx) = D_{d}(\xx;q,t) (q/t)^{d(k-1)}
\end{align*}
Hence we see that (\ref{eq}) is equivalent to
\begin{align*}
      \Delta^{-1}(\xx) x^{\mu} e(\xx;\ss)^{-1}  N\left(z (t/q)^{k-1}\right) e(\xx;\ss)x^{-\mu} \Delta(\xx) P_{\mu}(\xx;q,q/t) =D(z) P_{\mu}(\xx;q,t)
\end{align*}
and so
\begin{align*}
    D(z) P_{\mu}(\xx;q,t)= P_{\mu}(\xx;q,t) \prod_{i=1}^{k}\frac{\varphi(t q^{\mu_i} (t/q)^{i-1} z)}{\varphi(q^{\mu_i} (t/q)^{i-1} z)}
\end{align*}

\end{proof}

\begin{Note}
This result can be interpreted as an evaluation formula for a $q$-integral of Selberg type. Using $q$-integral notation, the above equation can be written as
\begin{multline*}
\int_{[0,\aa]} P_{\mu}(\xx;q,t) h(\xx) \prod_{i,j=1}^{k} \frac{\varphi\left( tx_j/x_i \right)}{\varphi\left( q x_j/x_i \right)}  \frac{\varphi\left( qx_j/a_i \right)}{\varphi\left( t x_j/a_i \right)}  d_q \xx \\
=P_{\mu}(\aa;q,t) \prod_{i=1}^{k}\frac{\varphi(t q^{\mu_i} (t/q)^{i-1} z)}{\varphi(q^{\mu_i} (t/q)^{i-1} z)}
\end{multline*}
where 
$$
h(\xx):=\exp\left(\frac{1}{\ln(q)} \ln(z) \ln(x_1\ldots x_k) \right)
$$
and the $q$-integral is defined as
$$
\int_{[0,\aa]} g(\xx) d_q \xx := \sum_{d_1,\ldots, d_k \geq 0} g(a_1 q^{d_1}, \ldots, a_k q^{d_k})
$$
The above $q$-integral resembles the $q$-integrals considered by Kaneko \cite{Kaneko} and Warnaar \cite{OWselb}.

In the special case when $\mu$ is the empty partition, we have 
\begin{multline*}
    \int_{[0,\aa]} \exp\left(\frac{1}{\ln(q)} \ln(z) \ln(x_1\ldots x_k) \right) \prod_{i,j=1}^{k} \frac{\varphi\left( tx_j/x_i \right)}{\varphi\left( q x_j/x_i \right)}  \frac{\varphi\left( qx_j/a_i \right)}{\varphi\left( t x_j/a_i \right)}  d_q \xx \\
    = \prod_{i=1}^{k} \frac{\varphi(t  (t/q)^{i-1} z)}{\varphi( (t/q)^{i-1} z)}
\end{multline*}
From the integral formula of the vertex function (\ref{qint}), it is straightforward to see that this is equal to the vertex function of the quiver variety $T^*Gr(k,k)$.
\end{Note}

\subsection{Proof of Theorem \ref{insertion}}
Fix $k,n, \lambda$ as before. The full flag variety inside of $\mathbb{C}^k$ can be described as a Nakajima quiver variety, see \cite{tRSKor}. We label the K\"ahler parameters of the flag variety by $z_{n-k+1},\ldots, z_{n-1}$, and we identify these with a subset of the K\"ahler parameters of $X^{!}$. We label the equivariant parameters by $s_1,\ldots,s_k,\hbar^{!}$. Let $\textbf{F}(\ss,\zz)$ be the vertex function of the flag variety at the fixed point $W_1 \subset W_2 \subset \ldots \subset W_k=\mathbb{C}^k$ where $W_i$ is spanned by the first $i$ coordinate vectors.  
\begin{Lemma}\label{flag}
Substituting the K\"ahler parameters and $\hbar^{!}$ with the map $\kappa$ gives
$$
\kappa\left(\textbf{F}(\ss;\zz) \right) = p(\aa;\ss;q,q/\hbar)
$$
where $\aa=(a_{n-k+1},\ldots,a_n)$ and $\ss=(s_1,\ldots,s_k)$.
\end{Lemma}
\begin{proof}
This can be seen directly by comparing the formula for $p(\aa;\ss;q,t)$ from \cite{NSmac} with the formula for $\textbf{F}(\ss,\zz)$ from \cite{dinksmir3}. Alternatively, this follows from Theorem 2.6 of \cite{tRSKor}, along with uniqueness of solutions to the bispectral problem (Theorem 1.1 of \cite{NSmac}).
\end{proof}
For the remainder of this subsection, we write $\mathsf{v}$ for the dimension vector of $X^{!}$, or equivalently, of $X_{\lambda}$. It was proven in \cite{dinksmir3} that:
\begin{Proposition}\label{vertexsum} Specializing the equivariant parameters to $s_i={\hbar^{!}}^{i-1} q^{d_{n-k,i}}$ in $\textbf{F}(\ss,\zz)$, we have
\begin{align*}
  \textbf{V}_{\lambda}(\zz)= \sum_{\substack{d_{i,j}}} \Psi \prod_{i=1}^{n-k} \prod_{j=1}^{\mathsf{v}_i} z_{i}^{d_{i,j}} \prod_{i=n-k+1}^{n-1} \prod_{j=1}^{\mathsf{v}_i} z_{i}^{d_{n-k,j}}  {\bf F}({\hbar^{!}}^{i-1} q^{d_{n-k,i}},\zz)
\end{align*}
where $\Psi \in \mathbb{C}(q,\hbar^{!})$ represents $\varphi$ function terms that do not depend on $d_{i,j}$ for $i>n-k$ and the summation is taken over $d_{i,j}$ for $1\leq i \leq n-k$, $1\leq j \leq \mathsf{v}_i$.
\end{Proposition}
From (\ref{vpt}), we see that the only terms that are nonzero arise from degree choices that satisfy
$$
d_{n-k,1}\leq \ldots \leq d_{n-k,k}
$$
So we define the partition $\mu$ by $\mu=(d_{n-k,k},\ldots, d_{n-k,1})$.

Then from (\ref{spec}) and Lemma \ref{flag}, we have
\begin{Lemma}[\cite{tRSKor}, Proposition 2.7]
$$
\aa^{\mu} \kappa\left(\textbf{F}\left(q^{\mu} {\hbar^{!}}^{\delta},\zz\right)\right) = P_{\mu}(\aa;q,q/\hbar)
$$
where $\aa=(a_{n-k+1},\ldots, a_{n})$, $\delta=(k-1,\ldots,0)$, and $\aa^{\mu}=a_{n-k+1}^{\mu_1} \ldots a_n^{\mu_{k}}$.
\end{Lemma}

Applying $\kappa$ to Proposition \ref{vertexsum} gives
\begin{Lemma}\label{sumform}
\begin{align*}
\kappa\left(\textbf{V}_{\lambda}(\zz) \right)& = \sum_{d_{i,j}} \kappa(\Psi) \prod_{i=1}^{n-k} \prod_{j=1}^{\mathsf{v}_i} \left(a_{i+1}/a_i\right)^{d_{i,j}}\prod_{i=1}^{k-1} \prod_{j=1}^{\mathsf{v}_i} \hbar^{-d_{i,j}} \prod_{i=n-k+1}^{n-1} \prod_{j=1}^{\mathsf{v}_i} \left(\hbar a_{i+1}/a_i\right)^{d_{n-k,j}} \\
& \aa^{-\mu} P_{\mu}(\aa;q,q/\hbar)
\end{align*}
\end{Lemma}

\begin{proof}[Proof of Theorem \ref{insertion}]
With our choice of notation, the operator $p_m$ in $D_{d}(\aa;q,\hbar)$ shifts $a_{n-k+m}$ by $q$. It is easy to see that 
$$
p_m \prod_{i=n-k+1}^{n-1} \prod_{j=1}^{\mathsf{v}_i} \left(\hbar a_{i+1}/a_i\right)^{d_{n-k,j}} =\prod_{i=n-k+1}^{n-1} \prod_{j=1}^{\mathsf{v}_i} \left(\hbar a_{i+1}/a_i\right)^{d_{n-k,j}} q^{\mu_m} p_m
$$
and 
$$
p_m \aa^{-\mu} =  \aa^{-\mu} q^{-\mu_m} p_m
$$
So the contributions of $q$ from these terms cancel. Using Lemma \ref{sumform}, this implies that 
$$
D_{d}(\aa;q,\hbar) \kappa\left( \textbf{V}_{\lambda}(\zz)\right) = \sum_{d_{i,j}} \left( \ldots \right) D_{d}(\aa;q,\hbar) P_{\mu}(\aa;q,q/\hbar)
$$
where $\left(\ldots \right)$ stands for the remaining terms in Lemma \ref{sumform}. By Theorem \ref{diagonal}, this implies 
\begin{align*}
D_{d}(\aa;q,\hbar) &\kappa\left( \textbf{V}_{\lambda}(\zz)\right) = \sum_{d_{i,j}} \left( \ldots \right) P_{\mu}(\aa;q,q/\hbar) \frac{(\hbar)_d}{(q)_d} P_{(d)}\left(q^{\mu_i}\left(\hbar/q\right)^{i-1};q,\hbar\right) \\
&= \sum_{d_{i,j}} \left( \ldots \right) P_{\mu}(\aa;q,q/\hbar) \frac{(\hbar)_d}{(q)_d} P_{(d)}\left(q^{\mu_i}\left(\hbar/q\right)^{i-k};q,\hbar\right) \left(q/\hbar\right)^{d(k-1)}
\end{align*}
where $\delta=(0,-1,\ldots,-k+1)$. The expression on the right is exactly equal to $\kappa$ applied to the insertion into $\textbf{V}_{\lambda}(\zz)$ of the descendant given in the statement of Theorem.
\end{proof}

\subsection{Proof of Theorem \ref{main}}

\begin{Proposition}\label{vgrcoeff}
\begin{align*}
\kappa\left(\textbf{V}_{\lambda}(\zz)\right)^{-1} D_{d}(\aa;q,\hbar) & \kappa\left( \textbf{V}_{\lambda}(\zz) \right) & \\
= \sum_{d_1+\ldots+d_k=d} & \prod_{i=1}^{n} \prod_{j=n-k+1}^n \frac{\left(\hbar a_j/a_i\right)_{d_j}}{\left(q a_j/a_i\right)_{d_j}} \prod_{i,j=n-k+1}^n \frac{\left(q a_j/a_i\right)_{d_j-d_i}}{\left(\hbar a_j/a_i\right)_{d_j-d_i}}
\end{align*}
\end{Proposition}
\begin{proof}
From Proposition \ref{prform}, we see that for $n-k+1\leq j \leq n$
$$
p_j^{d_j} \kappa\left(\textbf{V}_{\lambda}(\zz)\right) = \prod_{i=1}^{n-k} \frac{\left(\hbar a_j/a_i\right)_{d_j}}{\left(q a_j/a_i \right)_{d_j}}\kappa\left(\textbf{V}_{\lambda}(\zz)\right)
$$
which, along with Definition \ref{dop}, gives the result.
\end{proof}

We note that the right hand side in Proposition \ref{vgrcoeff} is precisely the coefficient of $\textbf{V}_{p}(z,\aa)$ corresponding to the choice $d_1,\ldots, d_k$.

\begin{proof}[Proof of Theorem \ref{main}]

Applying Proposition \ref{reducetopoint} gives
\begin{align*}
\kappa\left( \widetilde{\textbf{V}}^{!}_{p^{!}}(u,\zz) \right) &= \kappa\left( \textbf{V}_{\lambda}(\zz) \right)^{-1} \kappa\left( \textbf{V}^{!}_{p^{!}}(u,\zz)\right) \\
&= \kappa\left( \prod_{i=1}^k \frac{\varphi(u  {\hbar^{!}}^{i})}{\varphi(u q {\hbar^{!}}^{i-1})} \right) \kappa\left(\textbf{V}_{\lambda}(\zz)\right)^{-1}  \\ &\qquad \kappa\left(\textbf{V}_{\lambda}(\zz) \left\langle \sum_{d=0}^{\infty} \frac{\left(q/\hbar^{!}\right)_d}{(q)_d} P_{(d)}\left(\xx; q, q/\hbar^{!}\right) (\hbar^{!} u)^d  \right\rangle  \right) 
\end{align*}

By Theorem \ref{insertion}, this is equal to

\begin{multline*}
\kappa\left( \prod_{i=1}^k \frac{\varphi(u  {\hbar^{!}}^{i})}{\varphi(u q {\hbar^{!}}^{i-1})} \right) \kappa\left(\textbf{V}_{\lambda}(\zz)\right)^{-1} \sum_{d=0}^{\infty} D_{d}(\aa;q,\hbar) \kappa\left(\textbf{V}_{\lambda}(\zz) \left({\hbar^{!}}^k u \right)^d \right) \\
= \kappa\left( \prod_{i=1}^k \frac{\varphi(u  {\hbar^{!}}^{i})}{\varphi(u q {\hbar^{!}}^{i-1})} \right) \kappa\left(\textbf{V}_{\lambda}(\zz)\right)^{-1} \sum_{d=0}^{\infty} D_{d}(\aa;q,\hbar) \kappa\left(\textbf{V}_{\lambda}(\zz) \right) z^d
\end{multline*}
By Proposition \ref{vgrcoeff} this is just $\textbf{V}_p(\aa,z)$. By definition of $\kappa$,
$$
\kappa\left( \prod_{i=1}^k \frac{\varphi(u  {\hbar^{!}}^{i})}{\varphi(u q {\hbar^{!}}^{i-1})}  \right) = \prod_{i=1}^{k} \frac{\varphi\left( (\hbar/q)^{i-1} z\right)}{\varphi\left(\hbar (\hbar/q)^{i-1} z
\right)}
$$
So after multiplying this factor over, we obtain
$$
\widetilde{\textbf{V}}^{!}_{p^{!}} (u,\zz) = \widetilde{\textbf{V}}_p(\aa,z)
$$

\end{proof}

\printbibliography

@online{pcmilect,
      author         = "Okounkov, Andrei",
      title          = "{Lectures on K-theoretic computations in enumerative
                        geometry}",
      year           = "2015",
      eprint         = "1512.07363",
      archivePrefix  = "arXiv",
      primaryClass   = "math.AG",
      SLACcitation   = "%%CITATION = ARXIV:1512.07363;%%"
}

@article {EV,
    AUTHOR = {Etingof, P. and Varchenko, A.},
     TITLE = {Dynamical {W}eyl groups and applications},
   JOURNAL = {Adv. Math.},
  FJOURNAL = {Advances in Mathematics},
    VOLUME = {167},
      YEAR = {2002},
    NUMBER = {1},
     PAGES = {74--127},
      ISSN = {0001-8708},
     CODEN = {ADMTA4},
   MRCLASS = {17B10 (17B20 17B37 39A12 81R50)},
  MRNUMBER = {1901247 (2003d:17004)},
MRREVIEWER = {Anjan Kundu},
       DOI = {10.1006/aima.2001.2034}
}

@preamble{
   "\def\cprime{$'$} "
}

@online{OS,
      author         = "Okounkov, Andrei and Smirnov, Andrey",
      title          = "{Quantum difference equation for Nakajima varieties}",
      year           = "2016",
      eprint         = "1602.09007",
      archivePrefix  = "arXiv",
      primaryClass   = "math-ph",
    
}

@article {GL,
    AUTHOR = {Givental, Alexander and Lee, Yuan-Pin},
     TITLE = {Quantum {$K$}-theory on flag manifolds, finite-difference
              {T}oda lattices and quantum groups},
   JOURNAL = {Invent. Math.},
  FJOURNAL = {Inventiones Mathematicae},
    VOLUME = {151},
      YEAR = {2003},
    NUMBER = {1},
     PAGES = {193--219},
      ISSN = {0020-9910},
     CODEN = {INVMBH},
   MRCLASS = {14N35 (14C35 14M15 17B37 37K20 39A70 53D45)},
  MRNUMBER = {1943747 (2004g:14063)},
MRREVIEWER = {Domenico Fiorenza},
       DOI = {10.1007/s00222-002-0250-y},
       URL = {http://dx.doi.org/10.1007/s00222-002-0250-y},
}

@ARTICLE{OkBethe,
   author = {{Aganagic}, Mina and {Okounkov}, Andrei},
    title = "{Quasimap counts and Bethe eigenfunctions}",
  journal = {Mosc. Math. J.},
     year = 2017,
     volume = 17,
     pages = {565-600},
   adsurl = {http://adsabs.harvard.edu/abs/2017arXiv170408746A}
}

@online{AOElliptic,
       author = {{Aganagic}, Mina and {Okounkov}, Andrei},
        title = "{Elliptic stable envelopes}",
      journal = {arXiv e-prints},
         year = "2016",
        month = "4",
          eid = {arXiv:1604.00423},
archivePrefix = {arXiv},
       eprint = {1604.00423v4},
 primaryClass = {math.AG},
       adsurl = {https://ui.adsabs.harvard.edu/abs/2016arXiv160400423A}
}

@incollection {GinzburgLectures,
    AUTHOR = {Ginzburg, Victor},
     TITLE = {Lectures on {N}akajima's quiver varieties},
 BOOKTITLE = {Geometric methods in representation theory. {I}},
    SERIES = {S\'emin. Congr.},
    VOLUME = {24},
     PAGES = {145--219},
 PUBLISHER = {Soc. Math. France, Paris},
      YEAR = {2012},
   MRCLASS = {14L24 (16G20 17B67)},
  MRNUMBER = {3202703},
MRREVIEWER = {Xueqing Chen},
}

@article{Pushk1,
author = {Pushkar, Petr and Smirnov, Andrey and Zeitlin, Anton},
year = {2016},
month = {12},
title = {Baxter Q-operator from quantum K-theory},
volume = {360},
journal = {Adv. Math.}
}

@article {kirv,
	AUTHOR = {McGerty, Kevin and Nevins, Thomas},
	TITLE = {Kirwan surjectivity for quiver varieties},
	JOURNAL = {Invent. Math.},
	VOLUME = {212},
	YEAR = {2018},
	NUMBER = {1},
	PAGES = {161--187}
}

@online{MirSym1,
	author = {{Rim{\'a}nyi}, Rich{\'a}rd and {Smirnov}, Andrey and {Varchenko}, Alexand and {Zhou}, Zijun},
	title = "{3d Mirror Symmetry and Elliptic Stable Envelopes}",
	journal = {arXiv e-prints},
	year = "2019",
	month = 2,
	eid = {arXiv:1902.03677},
	pages = {arXiv:1902.03677},
	archivePrefix = {arXiv},
	eprint = {1902.03677},
	primaryClass = {math.AG},
	adsurl = {https://ui.adsabs.harvard.edu/abs/2019arXiv190203677R}
}

@article{MirSym2,
	author = {{Rim{\'a}nyi}, R. and {Smirnov}, A. and {Varchenko}, A. and {Zhou}, Z.},
	title = "{Three dimensional mirror self-symmetry of the cotangent bundle of the full flag variety}",
	journal = {SIGMA},
	volume = {15},
	year = "2019",
	month = 11,
	pages = {1-22}
}

@online{NakBow,
	author = {{Nakajima}, Hiraku and {Takayama}, Yuuya},
	title = "{Cherkis bow varieties and Coulomb branches of quiver gauge theories of affine type $A$}",
	journal = {arXiv e-prints},
	year = "2016",
	month = 6,
	eid = {arXiv:1606.02002},
	pages = {arXiv:1606.02002},
	archivePrefix = {arXiv},
	eprint = {1606.02002},
	primaryClass = {math.RT},
	adsurl = {https://ui.adsabs.harvard.edu/abs/2016arXiv160602002N}
}

@Book{ mac,
author = { Macdonald, I. G. },
title = { Symmetric functions and Hall polynomials},
isbn = { 0198535309 },
publisher = { Clarendon Press ; Oxford University Press Oxford : New York },
year = { 1979 },
type = { Book },
language = { English },
subjects = { Abelian groups.; Finite groups.; Hall polynomials.; Symmetric functions. },
life-dates = { 1979 -  },
catalogue-url = { https://nla.gov.au/nla.cat-vn2653038 },
}

@article{dinksmir,
      author = {{Dinkins}, Hunter and {Smirnov}, Andrey},
        title = "{Characters of tangent spaces at torus fixed points and $3d$-mirror symmetry}",
      journal = {Lett. Math. Phys.},
      pages = {to appear},
         year = "2019",
        month = "8",
archivePrefix = {arXiv},
      eprint = {1908.01199v2},
 primaryClass = {math.AG}
}

@article{qm,
title = {Stable quasimaps to GIT quotients},
journal = {J. Geom. Phys.},
volume = {75},
pages = {17 - 47},
year = {2014},
author = {Ionuţ Ciocan-Fontanine and Bumsig Kim and Davesh Maulik},
}

@online{KorZet,
	author = {{Koroteev}, Peter and {Zeitlin}, Anton M.},
	title = "{qKZ/tRS Duality via Quantum K-Theoretic Counts}",
	journal = {arXiv e-prints},
	year = "2018",
	month = 2,
	eid = {arXiv:1802.04463},
	pages = {arXiv:1802.04463},
	archivePrefix = {arXiv},
	eprint = {1802.04463},
	primaryClass = {math.AG},
	adsurl = {https://ui.adsabs.harvard.edu/abs/2018arXiv180204463K}
}

@online{tRSKor,
       author = {{Koroteev}, Peter},
        title = "{A-type Quiver Varieties and ADHM Moduli Spaces}",
      journal = {arXiv e-prints},
         year = "2018",
        month = "5",
          eid = {arXiv:1805.00986},
        pages = {arXiv:1805.00986},
archivePrefix = {arXiv},
       eprint = {1805.00986},
 primaryClass = {math.AG},
       adsurl = {https://ui.adsabs.harvard.edu/abs/2018arXiv180500986K}
}

@article{dinksmir2,
       author = {{Dinkins}, Hunter and {Smirnov}, Andrey},
        title = "{Quasimaps to zero-dimensional $A_{\infty}$-quiver varieties}",
      journal = {Int. Math. Res. Not. IMRN},
          pages={to appear},
         year = "2019",
        month = "12",
archivePrefix = {arXiv},
       eprint = {1912.04834},
 primaryClass = {math.AG}
}

@online{dinksmir3,
       author = {{Dinkins}, Hunter and {Smirnov}, Andrey},
        title = "{Capped vertex with descendants for zero dimensional $A_{\infty}$ quiver varieties}",
      journal = {arXiv e-prints},
         year = 2020,
        month = may,
          eid = {arXiv:2005.12980},
        pages = {arXiv:2005.12980},
archivePrefix = {arXiv},
       eprint = {2005.12980},
 primaryClass = {math.AG},
       adsurl = {https://ui.adsabs.harvard.edu/abs/2020arXiv200512980D}
}

@online{NSmac,
       author = {{Noumi}, Masatoshi and {Shiraishi}, Jun'ichi},
        title = "{A direct approach to the bispectral problem for the Ruijsenaars-Macdonald q-difference operators}",
      journal = {arXiv e-prints},
         year = 2012,
        month = jun,
          eid = {arXiv:1206.5364},
        pages = {arXiv:1206.5364},
archivePrefix = {arXiv},
       eprint = {1206.5364},
 primaryClass = {math.QA},
       adsurl = {https://ui.adsabs.harvard.edu/abs/2012arXiv1206.5364N}
}

@misc{NoumiSano,
       author = {{Noumi}, Masatoshi and {Sano}, Ayako},
        title = "{An infinite family of higher-order difference operators
that commute with Ruijsenaars operators of type $A$}",
     howpublished = {private communication},
         year = 2012,
        month = 7
}

@article{Kaneko,
     author = {Kaneko, Jyoichi},
     title = {$q$-Selberg integrals and Macdonald polynomials},
     journal = {Annales scientifiques de l'\'Ecole Normale Sup\'erieure},
     publisher = {Elsevier},
     volume = {Ser. 4, 29},
     number = {5},
     year = {1996},
     pages = {583-637},
     zbl = {0910.33011},
     mrnumber = {98k:33026}
}

@article{OWselb,
     author = {Warnaar, S. O.},
     title = {$q$-Selberg Integrals and Macdonald Polynomials},
     journal = {The Ramanujan Journal},
     volume = {10},
     number = {5},
     year = {2005},
     pages = {237-268}
}
\end{document}